\documentclass{book}
\usepackage{amsmath,amsthm,amssymb}
\usepackage{tmj}

\newcommand{\al}{\alpha}
\newcommand{\bt}{\beta}

\newcommand{\Dl}{\Delta}
\newcommand{\vf}{\varphi}
\newcommand{\om}{\omega}

\newcommand{\bC}{\mathbb{C}}
\newcommand{\bR}{\mathbb{R}}
\newcommand{\bQ}{\mathbb{H}}

\newcommand{\ol}{\overline}

\newcommand{\pa}{\partial}

\numberwithin{equation}{section}

\received{}
\revised{}
\accepted{}
\volumenumber{?}
\volumeyear{????}

\title{On the Differentiability of Quaternion Functions}{On the Differentiability of Quaternion Functions}

\author{Omar Dzagnidze}{Omar Dzagnidze}
\address{A. Razmadze Mathematical Institute of I. Javakhishvili Tbilisi State University, 2, University St., Tbilisi 0186, Georgia}
\email{odzagni@rmi.ge}

\begin{document}

\maketitle

\begin{abstract}
Motivated by the general problem of extending the classical theory of holomorphic
functions of a complex variable to the case of quaternion functions, we give a notion of an $\bQ$-\emph{derivative}
for functions of one quaternion variable. We show that the elementary quaternion functions introduced
by Hamilton as well as the quaternion logarithm function possess such a derivative. We conclude by establishing
rules for calculating $\bQ$-derivatives.
\end{abstract}

\metadata{30G35}{30A05, 30B10}{Quaternion, $\bQ$-derivative, elementary quaternion functions}

\section{Introduction}

    The importance of the theory of holomorphic (analytic) functions of one
complex variable  suggests looking for a similar theory for functions of three and more real variables. Considering the case of four variables, where the quaternion algebra $\bQ$ should replace the field $\bC$ of complex numbers, is, of course, especially natural.

Let us recall that the quaternion algebra was introduced by W.R. Hamilton in 1843 (see, e.g., \cite{Crowe 1967, Aleksandrova 1982}), and that, according to  {\em Frobenius' theorem} (see, e.g., \cite{Kantor and Solodovnikov 1973}), every finite-dimensional (associative) division algebra over the field $\bR$ of real numbers is isomorphic either to $\bR$, or to $\bC$, or to $\bQ$.

In a sense, there are three well-known methods for constructing the theory of holomorphic
functions of one complex variable: the derivative method, the polynomial method and the gradient method.
In the case of quaternion functions, none of them leads to a satisfactory quaternionic analogue
of the notion of a holomorphic function.

\subsection*{Derivative method} This method is based on the use of the limit definition of a derivative.
In the case of quaternion functions, such an approach yields at least two notions
of the quaternion derivative of a quaternion function $f(z)$. The point is that, as quaternions do not commute,
there are two different possibilities for regarding the ratio $\frac{f(z+h)-f(z)}{h}$:
one could replace it either by $[f(z+h)-f(z)] h^{-1}$ or by $h^{-1}[f(z+h)-f(z)]$.
This leads to the notion of a
right-hand derivative $$A(z)=\lim\limits_{h\to 0} [f(z+h)-f(z)] h^{-1}$$ and
to the notion of a left-hand derivative $$B(z)=\lim\limits_{h\to 0} h^{-1}[f(z+h)-f(z)] ,$$
provided that the corresponding one-sided limit exists.
Actually, both of these notions are too restricted. It turns out \cite{Krylov 1947,
Meylikhzon 1948, Buff 1973, Martinez 1975, Haddad 1981, 27} that only the functions $\vf(z)=az+b$
possess the right-hand derivative, and only the functions $\psi(z)=za+b$ possess the left-hand derivative,
while only the functions $\chi(z)=rz+b$ possess both left- and right-hand derivatives,
and they are equal. Here $a$ and $b$ are quaternions, while $r$ is a real number.

\subsection*{Polynomial method}
Let us consider a polynomial $P(x,y)\!=\!\sum\limits_{m,n} A_{m,n}x^my^n$ of two real
variables $x$, $y$ with complex coefficients $A_{m,n}=\al_{m,n}+i\bt_{m,n}$.
Replacing $x$ with $\frac{1}{2}(\ol{z}+z)$, and replacing $y$ with $\frac{1}{2}\,i(\ol{z}-z)$, we
obtain the polynomial $P^*(z,\ol{z})$ of the complex variables $z=x+iy$ and $\ol{z}=x-iy$. In order
that $P^*(z,\ol{z})$ be a polynomial of only the variable $z$, it is necessary and
sufficient that $P^*(z,\ol{z})$ satisfies the well-known Cauchy--Riemann condition.
This condition is precisely what is needed for distinguishing the class of polynomials
of the variable $z$.

    An analogous approach to the polynomials of the variables $x_0$, $x_1$,
$x_2$, $x_3$ with quaternion coefficients is too general to give a desired result
\cite{Sudbery 1979}. Indeed, the Hausdorff formulas $x_0=\frac{1}{4}(z-i_1zi_1-i_2zi_2-i_3zi_3)$, $x_1=\frac{1}{4i_1}(z-i_1zi_1+i_2zi_2+i_3zi_3)$, $x_2=\frac{1}{4i_2}(z+i_1zi_1-i_2zi_2+i_3zi_3)$ and $x_3=\frac{1}{4i_3}(z+i_1zi_1+i_2zi_2-i_3zi_3)$
\cite{Hausdorff 1900} allow us to express the real coordinates $x_k$ of the quaternion $z=x_0+x_1i_1+x_2i_2+x_3i_3$
 in terms of $z$ itself, without using the conjugate quaternion $\ol{z}=x_0-x_1i_1-x_2i_2-x_3i_3$. (Note that along
with the Hausdorff formulas, we also have the formulas $x_0=\frac{1}{2}(\ol{z}+z)$, $x_1=\frac{1}{2}(i_1\ol{z}-zi_1)$, $x_2=\frac{1}{2}(i_2\ol{z}-zi_2)$ and $x_3=\frac{1}{2}(i_3\ol{z}-zi_3)$, but unlike the complex case, they are not essential.) Hence
the functions which can be represented by quaternionic power series are just those which can be represented by power
series in four real variables.

\subsection*{Gradient method}
Loomann \cite{Loomann 1923} and Menchoff \cite{Menchoff 1935}
proved that any complex-valued continuous function that satisfies
the Cauchy--Riemann condition in a complex domain, is holomorphic in the
same domain\footnote{Tolstov \cite{Tolstov 1942} proved the some result replacing
continuity of a function by its boundedness.}
\cite{Saks 1937}. Thus, in order to extend the theory of holomorphic functions to the quaternionic setting,
one of the possible ways is to try to find a quaternionic analogue of the Cauchy-Riemann equations.
In 1935, Fueter\footnote{Dr. Rudolf Fueter is the
author of the book ``Synthetische Zahlentheorie'', Berlin and Leipzig, 1921, pp. VIII+271.}\cite{Fueter 1934--1935}
proposed such a quaternionic analogue by introducing two quaternion gradient operators
$$\frac{\pa^r}{\pa z}=\frac{\pa}{\pa x_0}+\frac{\pa}{\pa x_1}i_1+\frac{\pa}{\pa x_2}i_2+\frac{\pa}{\pa x_3}i_3$$
and
$$\frac{\pa^l}{\pa z}=\frac{\pa}{\pa x_0}+i_1\frac{\pa}{\pa x_1}+i_2\frac{\pa}{\pa x_2}+i_3\frac{\pa}{\pa x_3}.$$
He calls a quaternion function $f(z)$ of the quaternion variable $z=x_0+x_1i_1+x_2i_2+x_3i_3$
{\it right-regular}, provided
\begin{equation}\label{1-1}
   \frac{\pa^rf}{\pa z}=0.
\end{equation}
 {\it Left-regularity} of $f$ is defined in a similar way by requiring that
\begin{equation}\label{1-2}
  \frac{\pa^lf}{\pa z}=0.
\end{equation}
$f$ is called \emph{regular} if it is simultaneously
left and right regular. Such quaternion functions are to be thought of as the appropriate
generalization of holomorphic functions to the quaternionic setting.

    Any right or left regular quaternion function is harmonic, i.e.
satisfies Laplace's equation
$$
    \Dl\vf(z)=0, \quad \Dl=\frac{\pa^2}{\pa x_0^2}+\frac{\pa^2}{\pa x_1^2}+\frac{\pa^2}{\pa x_2^2}+\frac{\pa^2}{\pa x_3^2}\,.
$$ Moreover, for any real harmonic function, there is a regular function whose real part is exactly the
given function (see \cite{Fueter 1934--1935}). Thus there are plenty of regular functions.

    Although Fueter's approach is quite powerful and gives rise to a fully formed theory of regular functions,
it has some significant flaws. One such is the fact that even the functions $\psi_n(z)=z^n$  fail to be regular
(they are non-harmonic, since, for example, $\Dl(z^2)=-4$ and hence cannot be regular). Note that the functions $\Dl z^n$ are
regular \cite{Fueter 1934--1935}. Another flaw is that the class of regular functions of a
quaternion variable do not form an algebra in the same sense that the holomorphic functions do:
for example, regular functions  cannot be multiplied to give further regular functions.

\bigskip

    The aim of the paper is to propose a new definition of a derivative for quaternion functions of one quaternion
variable and show that all the elementary functions as well as the quaternion logarithm function possess such a derivative.
We also give rules for calculating such derivatives.

    We conclude Introduction with an incomplete list of references where
various applications of quaternions are discussed \cite{Arnol'd 1995, Branets and
Shmiglevski 1973, Dirac 1944--1945, 12, 15, Gurlebeck and Sprossig 1989,
Hoshi 1962, Kravchenko and Shapiro 1996, Linnik 1949, Morse and Feshbach
1953, Peterz 1969, 41, Taussky 1970, Vakhania 1998}.

\section{The notion of an $\bQ$-derivative}

   We begin by the following

\begin{definition}\label{d2.1}\em
A quaternion function $f(z)$, $z=x_0+x_1i_1+x_2i_2+x_3i_3$, defined on some
neighborhood $G \subset \bQ$ of a point $z^0=x_0^0+x_1^0i_1+x_2^0i_2+x_3^0i_3$, is called $\bQ$-\emph{differentiable at} $z^0$ if there exist
two sequences of quaternions $A_k(z^0)$ and $B_k(z^0)$ such that $\sum\limits_k A_k(z^0) B_k(z^0)$ is finite and that
the increment $f(z^0+h)-f(z^0)$ of the function $f(z)$ can be represented as
\begin{equation}\label{2-1}
f(z^0+h)-f(z^0)=\sum_k A_k(z^0)\cdot h\cdot B_k(z^0)+\om(z^0,h),
\end{equation}
where
\begin{equation}\label{2-2}
    \lim_{h\to 0} \frac{|\om(z^0,h)|}{|h|}=0
\end{equation} and $z^0+h \in G$.
In this case, the quaternion $\sum\limits_k A_k(z^0) B_k(z^0)$ is called the $\bQ$-\emph{derivative of the function $f$ at the point}
$z^0$ and is denoted $f'(z^0)$. Thus
\begin{equation}\label{2-3}
    f'(z^0)=\sum_k A_k(z^0)B_k(z^0).
\end{equation}
\end{definition}

    The uniqueness of the $\bQ$-derivative follows from the fat that the right-hand part of \eqref{2-3}, if it exists,
is just the partial derivative $f'_{x_0}(z^0)$ of $f(z)$ at $z^0$ with respect to its real variable.

    In the sequel, the symbol $o(h)$ denotes any function $\om(z^0,h)$ satisfying \eqref{2-2}.

\begin{remark}\label{r1.1}
Note that the same definition still makes perfectly good sense for any map between Banach algebras.
Moreover, all the proofs  of our results remain valid (except for Proposition \ref{p3.5}, which still remain
true if we take $\phi$ to be invertible in a neighborhood of $z^0$), since they require only those properties of $\bQ$
which are also possessed by any Banach algebra.
\end{remark}

\bigskip
The following functions introduced by Hamilton

\begin{align*}
    & z^n, \quad n=0,1,2,\dots\;, \\
    e^z & =1+z+\frac{z^2}{2!}+\frac{z^3}{3!}+\cdots \;,\\
    \cos z & =1-\frac{z^2}{2!}+\frac{z^4}{4!}-\cdots \;,\\
    \sin z & =z-\frac{z^3}{3!}+\frac{z^5}{5!}-\cdots
\end{align*}
are the \emph{basic elementary quaternion functions of one quaternionic variable~$z$}.

    Let us now show that the basic elementary functions are $\bQ$-differentiable.

\begin{proposition}\label{p2.2}$(z^n)'=nz^{n-1} \,\,\text{for}\,\, n=0,1,2,\dots\,, \,\,\text{and for } z \in \bQ.$
\end{proposition}

\begin{proof}
We first show that the following equality holds for $n=1,2,\dots$
\begin{align}
    (z+h)^n-z^n& =z^{n-1}h+z^{n-2}hz+z^{n-3}hz^2 \notag \\
    & \quad +\cdots+zhz^{n-2}+hz^{n-1}+o(h)  \label{2-5}
\end{align}
For $n=1$ it is obvious. Assuming now that it is  true for $n=k$,
we find
\allowdisplaybreaks
\begin{multline*}
    (z+h)^{k+1}-z^{k+1}=(z+h)(z+h)^k-z^{k+1}\\
    =(z+h)(z^k+z^{k-1}h+z^{k-2}hz+\cdots +zhz^{k-2}+hz^{k-1}+o(h))-z^{k+1}\\
    =z^{k+1}+z^kh+z^{k-1}hz+\cdots +z^2hz^{k-2}+zhz^{k-1}+hz^k+o(h))-z^{k+1}\\
    =z^kh+z^{k-1}hz+z^{k-2}hz^2+\cdots +z^2hz^{k-2}+zhz^{k-1}+hz^k+o(h).
\end{multline*}

It then follows from \eqref{2-3} and \eqref{2-5} that
    $$(z^n)'=z^{n-1}\cdot 1+z^{n-2}\cdot z+z^{n-3}\cdot z^2+\cdots +z\cdot z^{n-2}+1\cdot z^{n-1}=nz^{n-1}.$$ 
Thus by induction we have proved that $(z^n)'=nz^{n-1}$ for all $\quad n=0,1,2,\dots\,$
\end{proof}

   In order to proceed further, we need the following lemma.

\begin{lemma}\em
The following equalities and estimates are valid for $|h|<1$:
\begin{align*}
    \frac{(z+h)^2-z^2}{2!} & =\frac{zh+hz}{2!}+A_2, \\
    \frac{(z+h)^3-z^3}{3!} & =\frac{z^2h+zhz+hz^2}{3!}+A_3, \\
    \frac{(z+h)^4-z^4}{4!} & =\frac{z^3h+z^2hz+zhz^2+hz^3}{4!}+A_4, \\
    \frac{(z+h)^5-z^5}{5!} & =\frac{z^4h+z^3hz+z^2hz^2+zhz^3+hz^4}{5!}+A_5,
\end{align*}
and so on, where
\allowdisplaybreaks
\begin{align*}
    A_2& =\frac{1}{2!}\,h^2, \quad |A_2|<|h|^2, \quad A_2=o(h); \\
    A_3& =\frac{1}{3!}\,(zh^2+hzh+h^2z+h^3), \quad |A_3|<\frac{2^3}{3!}\,(|z|\,|h|^2+|h|^3) \\
    & < \begin{cases}
            \frac{2^3}{3!}\,|h|^2(1+|h|)<\frac{2^3}{3!}\,|h|^2\cdot \frac{1}{1-|h|} & \text{for} \;\; |z|<1, \\
            \frac{2^3}{3!}\,|z|\,|h|^2(1+|h|)<\frac{2^3}{3!}\,|z|\,|h|^2\cdot \frac{1}{1-|h|} & \text{for} \;\; |z|\geq 1;
        \end{cases} \notag \\
    A_4& =\frac{1}{4!}\,(z^2h^2+zhzh+zh^2z+zh^3+hz^2h+hzhz+hzh^2+h^2z^2 \\
    & \quad +h^2zh+h^3z+h^4)), \quad |A_4|<\frac{2^4}{4!}\,(|z|^2\,|h|^2+|z|\,|h|^3+|h|^4) \notag \\
    & < \begin{cases}
            \frac{2^4}{4!}\,|h|^2(1+|h|+|h|^2)<\frac{2^4}{4!}\,|h|^2\cdot \frac{1}{1-|h|} & \text{for} \;\; |z|<1, \\
            \frac{2^4}{4!}\,|z|^2\,|h|^2(1+|h|+|h|^2)<\frac{2^4}{4!}\,|z|^2\,|h|^2\cdot \frac{1}{1-|h|} & \text{for}\;\; |z|\geq 1;
        \end{cases} \notag \\
    A_5& =\frac{1}{5!}\,(z^3h^2+z^2hzh+z^2h^2z+z^2h^3+zhz^2h+zhzhz+zhzh^2\\
    & \quad +zh^2z^2 +zh^2zh+zh^3z+zh^4+hz^3h+hz^2hz+hz^2h^2 \notag \\
    & \quad +hzhz^2 +hzhzh+hzh^3z+hzh^3+h^2z^3+h^2z^2h+h^2zhz \notag \\
    & \quad +h^2zh^2+h^3z^2 +h^3zh+h^4z+h^5), \notag \\
    & \quad |A_5|<\frac{2^5}{5!}\,(|z|^3\,|h|^2+|z|^2\,|h|^3+|z|\,|h|^4+|h|^5) \notag \\
    & < \begin{cases}
            \frac{2^5}{5!}\,|h|^2(1+|h|+|h|^2+|h|^3)<\frac{2^5}{5!}\,|h|^2\cdot \frac{1}{1-|h|} & \text{for} \;\; |z|<1, \\
            \frac{2^5}{5!}\,|z|^3\,|h|^2(1+|h|+|h|^2+|h|^3)\\
            \qquad <\frac{2^5}{5!}\,|z|^3\,|h|^2\cdot \frac{1}{1-|h|}
                    & \text{for}\;\; |z|\geq 1,
        \end{cases} \notag
\end{align*}
and so on,
$$
    |A_n|<  \begin{cases}
                \frac{2^n}{n!}\,|h|^2\cdot \frac{1}{1-|h|} & \text{for} \;\; |z|<1, \\
                \frac{2^n}{n!}\,|z|^{n-2}|h|^2\cdot \frac{1}{1-|h|} & \text{for} \;\; |z|\geq 1.
            \end{cases}
$$
Therefore
$$
    \sum_{n=3}^\infty |A_n|<
            \begin{cases}
                |h|^2\cdot \frac{1}{1-|h|}\cdot \sum\limits_{n=3}^\infty \frac{2^n}{n!} & \text{for} \;\; |z|<1, \\
                |h|^2\cdot \frac{1}{1-|h|}\cdot \sum\limits_{n=3}^\infty \frac{2^n}{n!}\,|z|^{n-2} & \text{for} \;\; |z|\geq 1
            \end{cases}
$$
and the series $\sum\limits_{n=3}^\infty \frac{2^n}{n!}$ and $\sum\limits_{n=3}^\infty \frac{2^n}{n!}\,|z|^{n-2}$ are converging by virtue of the ratio test \emph{\cite{Marsden 1974}}.  Thus, $\sum\limits_{n=3}^\infty |A_n|=o(h)$ for any fixed finite quaternion $z$.
\end{lemma}

\begin{proposition}\label{p2.4}\em
We have the equality
$$
    (e^z)'=e^z.
$$
\end{proposition}

\begin{proof}
The equality $$ e^z =1+z+\frac{z^2}{2!}+\frac{z^3}{3!}+\cdots$$ implies that,
for any $h \in \bQ$,

$$e^{z+h}-z^z=h+\frac{(z+h)^2-z^2}{2!}+\frac{(z+h)^3-z^3}{3!}+\frac{(z+h)^4-z^4}{4!}+\cdots\;,$$
and applying Lemma 2.3 to the right-hand side of this equality, we obtain
\begin{align*}
    e^{z+h}-e^z =h &+ \frac{1}{2!}\,(zh+hz) \\
    & +\frac{1}{3!}\,(z^2h+zhz+hz^2) \\
    & +\frac{1}{4!}\,(z^3h+z^2hz+zhz^2+hz^3) \\
    & +\cdots+o(h).
\end{align*}
Therefore
\begin{multline}\label{2-19}
    e^{z+h}-e^z=\left(1+\frac{z}{2!}+\frac{z^2}{3!}+\cdots\right)h \\
    +\left(\frac{1}{2!}+\frac{z}{3!}+\frac{z^2}{4!}+\cdots\right)hz+
        \left(\frac{1}{3!}+\frac{z}{4!}+\frac{z^2}{5!}+\cdots\right)hz^2\\
        +\cdots+o(h)
\end{multline}
and hence
\begin{align*}
    (e^z)'=1+\frac{z}{2!}+\frac{z^2}{3!}+\frac{z^3}{4!}+\cdots  & \\
    +\frac{z}{2!}+\frac{z^2}{3!}+\frac{z^3}{4!}+\cdots & \\
    + \frac{z^2}{3!}+\frac{z^3}{4!}+\cdots & \\
    + \frac{z^3}{4!}+\cdots &  \\
    =1+2\,\frac{z}{2!}+3\,\frac{z^2}{3!}+4\,\frac{z^3}{4!}+\cdots & \\
    =1+\,\frac{z}{1!}+\frac{z^2}{2!}+\frac{z^3}{3!}+\cdots & = e^z. 
\end{align*}
\end{proof}

\begin{proposition}\label{p2.5}\em
The equality
$$
    (\sin z)'=\cos z
$$
is valid.
\end{proposition}

\begin{proof}
\allowdisplaybreaks[0]
\begin{align*}
    \sin (z+h)& -\sin z\\
    & =(z+h)-\frac{(z+h)^3}{3!}+\frac{(z+h)^5}{5!}-\cdots -z+\frac{z^3}{3!}-\frac{z^5}{5!}+\cdots \\
    & =h-\frac{(z+h)^3-z^3}{3!}+\frac{(z+h)^5-z^5}{5!}- \cdots \\
    & =h-\frac{1}{3!}\,(z^2h+zhz+hz^2)\\
    & \quad +\frac{1}{5!}\,(z^4h+z^3hz+z^2hz^2+zhz^3+hz^4)+\cdots+o(h).
\end{align*}
Therefore
\begin{multline*}
    \sin(z+h)-\sin z=h+\left(-\frac{z^2}{3!}+\frac{z^4}{5!}\right)h \\
        +zh\left(-\frac{z}{3!}+\frac{z^3}{5!}\right)+h\left(-\frac{z^2}{3!}+\frac{z^4}{5!}\right)+\cdots+o(h).
\end{multline*}
Hence
\allowdisplaybreaks
\begin{align*}
    (\sin z)' & =1-\frac{z^2}{3!}+\frac{z^4}{5!}+z\left(-\frac{z}{3!}+\frac{z^3}{5!}\right)-\frac{z^2}{3!}+\frac{z^4}{5!}+\cdots\\
    & =1-\frac{z^2}{3!}+\frac{z^4}{5!}-\frac{z^2}{3!}+\frac{z^4}{5!}-\frac{z^2}{3!}+\frac{z^4}{5!}+\cdots\\
    & =1-\frac{z^2}{2!}+\frac{z^4}{4!}-\cdots=\cos z. 
\end{align*}
\end{proof}

    Similarly, one has

\begin{proposition}\label{p2.6}\em
The equality
$$
    (\cos z)'=-\sin z
$$
is valid.
\end{proposition}

\section{Rules for calculating $\bQ$-derivatives}

    The rules for calculating $\bQ$-derivatives are identical to those derived in a standard calculus course.

\begin{proposition}\label{p3.1}\em
Let $f$ and $\vf$ be two functions defined on a neighborhood of $z^0 \in \bQ$.
If both $f$ and $\vf$ are $\bQ$-differentiable at $z^0$, then
\begin{enumerate}
  \item  both $c f$ and $f c$ are $\bQ$-differentiable at $z^0$ for all $c \in \bQ$ and $(c f)'(z^0)=c f'(z^0)$ and $(fc)'(z^0)=f'(z^0) c$;
  \item  $f+\vf$ is $\bQ$-differentiable at $z^0$ and $(f+\vf)'(z^0)=f'(z^0)+\vf'(z^0)$;
  \item $f \vf$ is $\bQ$-differentiable at $z^0$ and $(f \vf)'(z^0)=f'(z^0)\vf(z^0)+f(z)\vf'(z^0)$.
\end{enumerate}
\end{proposition}

\begin{proof}The proof of 1. is obvious.

Since $f$ and $\vf$ are $\bQ$-differentiable at $z^0$, there are representations
\begin{align*}
    f(z^0+h)-f(z^0) & =\sum_k A_k hB_k+o(h), \\
    \vf(z^0+h)-\vf(z^0) & =\sum_k C_k hD_k+o(h).
\end{align*}
Then
\begin{align*}
    (f+\vf)(z^0+h)-(f+\vf)(z^0) & =[f(z^0+h)-f(z^0)]+[\vf(z^0+h)-\vf(z^0)] \\
    & =    \sum_k A_k hB_k+\sum_k C_khD_k+ o(h),
\end{align*}
and hence
$$
    (f+\vf)'(z^0)= \sum_k A_k B_k+\sum_k C_k D_k=f'(z^0)+\vf'(z^0).
$$ This proves 2.

Next, since
\begin{align*}
    f(z^0+h)& \vf(z^0+h)-f(z^0)\vf(z^0)= \\
    & =[f(z^0+h)-f(z^0)]\vf(z^0+h)+f(z^0)[\vf(z^0+h)-\vf(z^0)] \\
    & =\bigg[\sum_k A_k h B_k+o(h)\bigg]\vf(z^0+h)+f(z^0)\bigg[ \sum_k C_k h D_k+o(h)\bigg] \\
    & =\bigg[\sum_k A_k h B_k+o(h)\bigg]\cdot \bigg[\vf(z^0)+ \sum_k C_k h D_k+o(h)\bigg]\\
    & \quad +f(z^0)\bigg[ \sum_k C_k h D_k+o(h)\bigg] \\
    & =\bigg( \sum_k A_k h B_k\bigg)\vf(z^0)+f(z^0) \sum_k C_k h D_k+o(h),
\end{align*}
it follows that
\[
    (f \vf)'(z^0)=\bigg( \sum_k A_k  B_k\bigg)\vf(z^0)+f(z^0) \sum_k C_k  D_k= f'(z^0)\vf(z^0)+f(z^0)\vf'(z^0),
\] proving 3.
\end{proof}

The following two corollaries are immediate.

\begin{corollary}\label{c3.3}\em
If $f_1,f_2,\dots,f_n$ are $\bQ$-differentiable functions at a point $z^0$, then
their product $f_1 f_2\cdots f_n$ is also $\bQ$-differentiable at $z^0$ and we have:
\begin{align*}
(f_1 f_2\cdots f_n)'(z^0)
&=f_1'(z^0) f_2(z^0) \cdots f_n(z^0)+f_1(z^0) f_2'(z^0) f_3(z^0)\cdots  f_n(z^0)+\\
&\cdots+f_1(z^0) \cdots f_{n-1}(z^0) f_n'(z^0).
\end{align*}
\end{corollary}

\begin{corollary}\label{c3.4}\em
If a function $f$ is $\bQ$-differentiable at a point $z^0$, then $f^n$ is also $\bQ$-differentiable at $z^0$ for all $n=1,2,\dots$
and we have:
$$
    (f^n)'(z^0)=f'(z^0) f^{n-1}(z^0)+f(z^0) f'(z^0) f^{n-2}(z^0)+\cdots +f^{n-1}(z^0) f'(z^0).
$$
\end{corollary}

\begin{proposition}\label{p3.5}\em
If a function $\vf$ is $\bQ$-differentiable at a point $z^0$ and if $\vf\neq 0$ in a neighborhood of
$z^0$, then\footnote{For each quaternion
$q\neq 0$, there is a (unique) quaternion $\frac{1}{q}$, called the \emph{inverse} of $q$, for which $q\cdot \frac{1}{q}=1=\frac{1}{q}\cdot q$.
The inverse of $q$ is sometimes
denoted by the symbol $q^{-1}$.} $\frac{1}{\vf}$ is also $\bQ$-differentiable at $z^0$ and we have:
$$
    \left(\frac{1}{\vf}\right)'(z^0)=-\frac{1}{\vf(z^0)}\cdot \vf'(z^0) \cdot \frac{1}{\vf(z^0)}.
$$
\end{proposition}

\begin{proof}
We first observe that for any two nonzero quaternions $q_1$ and $q_2$, the following equality
$$q_1^{-1}-q_2^{-1}=q_1^{-1}(q_1 -q_2)q_2^{-1}(q_1-q_2)q_2^{-1}-q_2^{-1}(q_1 -q_2)q_2^{-1}$$
holds. Indeed, using that $q_1^{-1}$ is the inverse of $q_1$ and $q_2^{-1}$ is the inverse of $q_2$, we obtain
\begin{align*}
q_1^{-1}(q_1 -q_2)&q_2^{-1}(q_1-q_2)q_2^{-1}-q_2^{-1}(q_1 -q_2)q_2^{-1}=\\
&=(1-q_1^{-1}q_2)q_2^{-1}(q_1q_2^{-1}-1)-(q_2^{-1}q_1-1)q_2^{-1}\\
&=(q_2^{-1}-q_1^{-1})(q_1q_2^{-1}-1)-(q_2^{-1}q_1 q_2^{-1}-q_2^{-1})\\
&=(q_2^{-1}q_1 q_2^{-1}-q_2^{-1} -q_2^{-1}+q_1^{-1})-(q_2^{-1}q_1 q_2^{-1}-q_2^{-1}),
\end{align*} as needed.

Putting $\vf(z^0+h)$ and $\vf(z^0)$ in the equality, we obtain
\begin{multline*}
    \frac{1}{\vf(z^0+h)}-\frac{1}{\vf(z^0)}=\\
    = \left\{ -\frac{1}{\vf(z^0)}+\frac{1}{\vf(z^0+h)}\,[\vf(z^0+h)-\vf(z^0)]\,\frac{1}{\vf(z^0)}\right\} \cdot
        [\vf(z^0+h)-\vf(z^0)]\,\frac{1}{\vf(z^0)}\,.
\end{multline*}
We now calculate
\begin{multline*}
    \frac{1}{\vf(z^0+h)}-\frac{1}{\vf(z^0)} = -\frac{1}{\vf(z^0)}\,[\vf(z^0+h)-\vf(z^0)]\,\frac{1}{\vf(z^0)}\\
    + \frac{1}{\vf(z^0+h)}\,[\vf(z^0+h)-\vf(z^0)]\,\frac{1}{\vf(z^0)}\, [\vf(z^0+h)-\vf(z^0)]\,\frac{1}{\vf(z^0)} \\
    = -\frac{1}{\vf(z^0)}\,\left[ \sum C_k h D_k+o(h) \right] \frac{1}{\vf(z^0)}+o(h) \\
    = -\frac{1}{\vf(z^0)}\,\left[ \sum C_k h D_k \right] \frac{1}{\vf(z^0)}+o(h).
\end{multline*}
Hence
\[
    \left(\frac{1}{\vf}\right)'(z^0)=-\frac{1}{\vf(z^0)}\,\left[ \sum C_k  D_k \right] \frac{1}{\vf(z^0)}=
        -\frac{1}{\vf(z^0)}\cdot\vf'(z^0) \cdot \frac{1}{\vf(z^0)}\,. 
\]
\end{proof}

\begin{corollary}\label{c3.6}\em
For $z\neq 0$ we have:
$$
    (z^m)'=mz^{m-1}, \quad m=-1,-2,\dots\,.
$$
\end{corollary}

\begin{proof} Putting $n=-m$ and using Propositions \ref{p2.2} and \ref{p3.5}, we obtain

$$
    (z^m)'=\left(\frac{1}{z^n}\right)'=-\frac{1}{z^n}\,(z^n)'\,\frac{1}{z^n} =-\frac{1}{z^n} \, nz^{n-1}\,\frac{1}{z^n}=
        -nz^{-n-1}=mz^{m-1}\,.
$$

\end{proof}

\begin{corollary}\label{c3.7}\em
For an arbitrary constant $c $, we have:
$$
    \left(\frac{1}{c-z}\right)'=\frac{1}{(c-z)^2}, \quad z\neq c.
$$
\end{corollary}

\begin{corollary}\label{c3.8}\em
If quaternionic functions $f$ and  $\vf$ are $\bQ$-differentiable at a point $z^0$ and $\vf\neq 0$ in a neighborhood of
$z^0$, then the functions $f\cdot \frac{1}{\vf}$ and $\frac{1}{\vf}\cdot f$ are also $\bQ$-differentiable at $z^0$ and we have:
\begin{align*}
    \left(f\cdot \frac{1}{\vf}\right)'(z^0) & =f'(z^0)\cdot \frac{1}{\vf(z^0)} -f(z^0)\,\frac{1}{\vf(z^0)} \cdot
        \vf'(z^0)\cdot \frac{1}{\vf(z^0)}\,
\end{align*}
and
\begin{align*}
    \left(\frac{1}{\vf}\cdot f \right)'(z^0) & =-\frac{1}{\vf(z^0)}\cdot \vf'(z^0)\cdot \frac{1}{\vf(z^0)} f(z^0)+
        \frac{1}{\vf(z^0)}\cdot f'(z^0)\,.
\end{align*}
\end{corollary}

\begin{proposition}\label{p3.9}\em
Let a function $f(z)$ be defined on some neighborhood of a point $z^0 \in \bQ$ and
let a function $F(w)$ be defined on some neighborhood of the point
$w^0=f(z^0)$.  Assume that $f$ is $\bQ$-differentiable at $z^0$ and that $F$ is $\bQ$-differentiable at $w^0$.
If $F'(w^0)=\sum\limits_k A_k B_k$, then the composite $Ff$ is $\bQ$-differentiable at $z^0$ and we have:
$$\left( Ff\right)'(z^0)=\sum_k A_k f'(z^0) B_k.$$
\end{proposition}

\begin{proof}
Let $z$ be in the neighborhood of $z^0$. Put $w=f(z)$. Then
\begin{align*}
    F(w)-F(w^0) & =\sum_k A_k (w-w^0)B_k+\om_1(w^0,w), \\
    f(z)-f(z^0) & =\sum_j C_j(z-z^0) D_j +\om_2(z^0,z),
\end{align*}
and using these presentations, we calculate
\begin{align*}
    F(f(z))& -F(f(z^0))  =\sum_k A_k (f(z)-f(z^0)) B_k +\om_1(f(z^0),f(z)) \\
    & =\sum_k A_k \bigg( \sum_j C_j(z-z^0)D_j \bigg) B_k +o(h)+\om_1(f(z^0),f(z)) \\
    & =\sum_k \sum_j A_k C_j(z-z^0)D_j B_k +o(h)+\om_1(f(z^0),f(z)).
\end{align*}
But since
$$
    \frac{|\om_1(f(z^0),f(z))|}{|z-z^0|} =\frac{|\om_1(f(z^0),f(z))|}{|w-w^0|} \cdot
        \frac{|z-z^0|}{|w-w^0|} \to 0, \quad z\to z^0,
$$
we have
\begin{align*}
    \left( Ff\right)'(z^0) & =\sum_k \sum_j A_k C_j D_j B_k=
        \sum_k A_k \bigg(\sum_j C_j D_j\bigg) B_k\\
    & =\sum_k A_k f'(z^0) B_k. 
\end{align*}
\end{proof}

Specializing the proposition to the case where $F(w)=w^n$ and applying \eqref{2-5} we get

\begin{corollary}\label{c3.10}\em
If a function $f$ is $\bQ$-differentiable, then
$$(f^n)'=f^{n-1}\cdot f'+f^{n-2}\cdot f'\cdot f+f^{n-3}\cdot f'\cdot f^2+\cdots +f'\cdot f^{n-1}.$$

\end{corollary}

\section{The $\bQ$-derivative of the quaternion logarithm function}

    A quaternion $w$ is called the \emph{logarithm} of a finite quaternion $z\neq 0$
if $z=e^w$, in which case we write $w=\ln z$.

    In order to define the $\bQ$-derivative $w'=(\ln z)'$, we first note that
the $\bQ$-derivative of the left-hand side of the identity $z=e^{\ln z}$ exits and is $1$
by Proposition \ref{p2.2}. Applying now Proposition \ref{p3.9} to the right-hand side and taking into account
\eqref{2-19}, we get
\begin{align}\label{4-1}
    1 & =\left(1+\frac{w}{2!}+\frac{w^2}{3!}+\cdots\right) \cdot w'+
        \left(\frac{1}{2!}+\frac{w}{3!}+\frac{w^2}{4!}+\cdots\right) \cdot w'\cdot w \\
    & \quad + \left(\frac{1}{3!}+\frac{w}{4!}+\frac{w^2}{5!}+\cdots\right) \cdot w'\cdot w^2+\cdots\;. \notag
\end{align}

    Thus, the $\bQ$-derivative $w'=(\ln z)'$ satisfies Equality \eqref{4-1}.

\begin{remark}\label{r4.1}
If $ww'$ and $w'w$ were equal, then we could write $w\cdot w',w^2 \cdot w',\dots$ instead of $w'\cdot w,w'\cdot w^2,\dots\;$,
and then Equality \eqref{4-1} would take the form
\begin{align*}
    1 & =\left(1+\frac{w}{2!}+\frac{w^2}{3!}+\cdots\right) \cdot w'+
        \left(\frac{w}{2!}+\frac{w^2}{3!}+\cdots\right) \cdot w' \\
    & \quad + \left(\frac{w^2}{3!}+\frac{w^3}{4!}+\frac{w^4}{5!}+\cdots\right) \cdot w'+\cdots \\
    & =\left(1+w+\frac{w^2}{2!}+\frac{w^3}{3!}+\cdots\right) \cdot w'=e^w\cdot w'=e^{\ln z}\cdot (\ln z)'=z\cdot (\ln z)'.
\end{align*}
So, we would obtain the classical formula
$$
    (\ln z)'=\frac{1}{z}\,,
$$
that is well known in the case of a complex variable $z$.
\end{remark}

\label{lastpage}

\end{document}